\documentclass{amsart}

\usepackage{a4wide,amssymb,amsthm,amsmath,mathrsfs}
\usepackage[colorlinks]{hyperref}

\theoremstyle{plain}

\newtheorem{theorem}{Theorem}
\newtheorem{lemma}{Lemma}
\newtheorem{proposition}{Proposition}
\newtheorem{corollary}{Corollary}

\theoremstyle{definition}
\newtheorem{definition}{Definition}

\theoremstyle{remark}

\newtheorem{remark}{Remark}

\def \M{\mathcal{M}}
\def \S{\mathcal{S}}
\def \g{\bar{g}}
\def \tangsp{\mathfrak{X}(\S)}
\def \normsp{\mathfrak{X}(\S)^\perp}

\def \tr{\textsf{tr}}
\def \Im{\textsf{Im}}
\def \dim{\textsf{dim}}

\def \tildeh{\widetilde{h}}
\def \tildeA{\widetilde{A}}
\def \id{\mathbf{1}}

\def \U{\mathscr{U}}

\begin{document}

\title[Umbilical spacelike submanifolds of arbitrary co-dimension]
{Umbilical spacelike submanifolds\\of arbitrary co-dimension}

\author[N. Cipriani]{Nastassja Cipriani}
\address{KU Leuven, Department of Mathematics, Celestijnenlaan 200B -- Box 2400, BE-3001 Leuven, Belgium and F\'isica Te\'orica, Universidad del Pa\'is Vasco, Apartado 644, 48080 Bilbao, Spain} 
\email{nastassja.cipriani@wis.kuleuven.be}

\author[J.M.M. Senovilla]{Jos\'e M. M. Senovilla}
\address{F\'isica Te\'orica, Universidad del Pa\'is Vasco, Apartado 644, 48080 Bilbao, Spain}
\email{josemm.senovilla@ehu.es}

\thanks{This work is partially supported by the Belgian Interuniversity Attraction Pole P07/18 (Dygest) and by the KU Leuven Research Fund project 3E160361 ``Lagrangian and calibrated submanifolds". NC and JMMS are supported under grant FIS2014-57956-P (Spanish MINECO--Fondos FEDER) and project IT956-16 of the Basque Government. JMMS is also supported under project UFI 11/55 (UPV/EHU)}

\begin{abstract}
Given a semi-Riemannian manifold, we give necessary and sufficient conditions for a Riemannian submanifold of arbitrary co-dimension to be umbilical along normal directions. We do that by using the so-called \emph{total shear tensor}, i.e., the trace-free part of the second fundamental form.\\
We define the \emph{shear space} and the \emph{umbilical space} as the spaces generated by the total shear tensor and by the umbilical vector fields, respectively. We show that the sum of their dimensions must equal the co-dimension.

\end{abstract}



\maketitle

\section{Introduction}

The notions of umbilical point and umbilical submanifold are classical in differential geometry. They have mainly been studied in the Riemannian setting and, apart from very few exceptions, they have been applied to submanifolds of co-dimension one (hypersurfaces). In \cite{CSV2016} the authors studied these concepts in a slightly more general framework, namely, allowing the ambient manifold to have arbitrary signature while requiring the submanifold to be \emph{spacelike} (i.e.\! endowed with a Riemannian induced metric). Motivated by the applications to gravitation and general relativity, in \cite{CSV2016} the authors focused on spacelike submanifolds of \emph{co-dimension two}. Notice that the results presented in \cite{CSV2016} generalized those presented in \cite{Senovilla 2012}. In the latter, the study of umbilical \emph{surfaces} ($2$-dimensional) in $4$-dimensional Lorentzian manifolds was carried out.

In the present paper we generalize what has been done in \cite{Senovilla 2012} and \cite{CSV2016}. We consider spacelike submanifolds of \emph{arbitrary co-dimension} and give a characterization of those which are umbilical with respect to some normal directions. Observe that when the co-dimension is one, the normal bundle is $1$-dimensional and thus the submanifold can only be umbilical along the unique normal direction. On the other hand, when the co-dimension is higher than one, there are several possibilities and the submanifold can be umbilical with respect to some normal vectors but non-umbilical with respect to others.

In order to characterize umbilical spacelike submanifolds we make use of the so-called \emph{total shear tensor}. This is defined as the trace-free part of the second fundamental form and it appears often in the mathematical literature, especially in conformal geometry. Nevertheless, it had never been given a name prior to \cite{CSV2016} and, surprisingly, its relationship with the umbilical properties of submanifolds seemed to be almost unknown or is not explicitly mentioned at least.

We introduce the notions of \emph{shear space} and \emph{umbilical space}. They are defined as the space generated by the image of total shear tensor and the one generated by the umbilical vector fields, respectively. They both belong to the normal bundle and they happen to be mutually orthogonal. We show that the existence of umbilical directions ``shrinks'' the shear space reducing its dimension. More precisely, the dimensions of the shear space and the umbilical space are linked in such a way that the sum of the two must equal the co-dimension of the submanifold. Moreover, in specific situations --for instance when the ambient manifold is Riemannian-- the direct sum of the two spaces generate the whole normal bundle.

The plan of the paper is as follows. In Section 2 we recall some basic concepts of submanifold theory, we introduce the shear objects and give the definitions of umbilical point, umbilical submanifold and umbilical space. In Section 3 we show how the shear space and the umbilical space are related, we present necessary and sufficient conditions for the submanifold to be umbilical and give some final remarks.

\section{Preliminaries}

\subsection{Basic concepts of submanifold theory}

We consider an orientable $n$-dimensional \emph{spacelike} submanifold $(\S, g)$ of a semi-Riemannian manifold $(\M,\g)$ with inmersion $\Phi: \S \longrightarrow \M$ and co-dimension $k$. Hence, $g:= \Phi^{\star}\g $ is positive definite everywhere on $\S$, so that $(\S,g)$ is in particular an oriented Riemannian manifold. Let $\tangsp$ and $\normsp$ denote the set of tangent and normal vector fields, respectively, on $\S$. The classical formulas of Gauss and Weingarten provide the decomposition of the vector field derivatives into their tangent and normal components \cite{KobayashiNomizu,Kriele,ONeill} as
\begin{align*}
& \overline{\nabla}_X Y = \nabla_X Y + h(X,Y), \hspace{1cm} \forall X,Y \in \tangsp\\
& \overline{\nabla}_X \xi = -A_{\xi}X + \nabla^{\perp}_X \xi, \hspace{1cm} \forall \xi \in \normsp, \, \, \forall X\in \tangsp
\end{align*}
where $\overline{\nabla}$ and $\nabla$ are the Levi-Civita connections of $(\M,\g)$ and $(\S,g)$ respectively, $h$ is the \emph{second fundamental form} or \emph{shape tensor} of the immersion and $A_{\xi}$ the \emph{Weingarten operator} relative to $\xi$. The derivation $\nabla^{\perp}$ so defined determines a connection on the normal bundle, $h(X,Y)=h(Y,X) \in \normsp$ for all $X,Y \in \tangsp$ and acts linearly (as a 2-covariant tensor) on its arguments while $A_{\xi}$ is self-adjoint for every $\xi \in \normsp$. The following relation holds
\begin{align*}
g(A_{\xi}X,Y) = \g(h(X,Y),\xi), \hspace{1cm} \forall X,Y \in \tangsp, \, \, \forall \xi \in \normsp .
\end{align*}
The \emph{mean curvature vector field} $H\in\normsp$ is $1/n$ times the trace (with respect to $g$) of $h$, so that for instance one can write \cite{KobayashiNomizu,Kriele,ONeill}
\begin{align*}
H = \frac{1}{n} \sum_{i=1}^n h(e_i,e_i)
\end{align*}
where $\{e_1,\ldots,e_n\}$ denotes an orthonormal frame on $\tangsp$.

\subsection{The total shear tensor and the shear operators}

Using the previous notations and conventions, the following definition is taken from \cite{CSV2016}
\begin{definition}
The \emph{total shear tensor} $\tildeh$ is defined as the trace-free part of the second fundamental form:
\begin{align*}
\tildeh(X,Y)= h(X,Y) - g(X,Y) H.
\end{align*}
The \emph{shear operator} associated to $\xi \in \normsp$ is the trace-free part of the corresponding shape operator:
\begin{align*}
\tildeA_{\xi} = A_{\xi} - \frac 1n \tr \tildeA_\xi \id
\end{align*}
where $\id$ denotes the identity operator.
\end{definition}

The total shear tensor and shear operators are obviously related by 
\begin{align}\label{relationship between any tildeA and tildeh}
g(\tildeA_{\xi}X,Y) = \g(\tildeh(X,Y),\xi), \hspace{1cm} \forall X,Y \in \tangsp \quad \forall\xi\in\normsp.
\end{align}

To the authors' knowledge, the trace-free part of the second fundamental form had never been given a name prior to \cite{CSV2016}. Nevertheless, it is easy to find it in the literature, in Riemannian settings, especially in connection with the conformal properties of submanifolds. A pioneer analysis appears in \cite{Fialkow 1944}, where an extensive exposition concerning conformal invariants was given. The total shear tensor is also in the basis of the definition of the so-called generalized Willmore functional \cite{Willmore 1988}. 

Denote by $\{\xi_1,\ldots,\xi_k\}$ a local frame in $\normsp$. With respect to this frame, there exist $k$ shear operators $\tildeA_1,\ldots,\tildeA_k$ such that the total shear tensor $\tildeh$ decomposes as
\begin{align}\label{decomposition formula for tildeh in a general basis}
\tildeh(X,Y) = \sum_{i=1}^k g(\tildeA_i X,Y) \xi_i , \hspace{1cm} \forall X,Y \in \tangsp .
\end{align}
If $\{\xi_1,\ldots,\xi_k\}$ is orthonormal, i.e. $\g(\xi_i,\xi_j)=\epsilon_i \delta_{ij}$ with $\epsilon_i^2=1$, then $\tildeA_i=\epsilon_i \tildeA_{\xi_i}$ for all $i$. However, in general, $\tildeA_i$ does not need be proportional to $\tildeA_{\xi_i}$, rather being a linear combination of $\tildeA_{\xi_1},\ldots,\tildeA_{\xi_k}$.

Given any normal vector field $\eta\in\normsp$, by (\ref{decomposition formula for tildeh in a general basis}) its corresponding shear operator $\tildeA_\eta$ can be expressed in terms of $\tildeA_1,\ldots,\tildeA_k$. Indeed, formulas (\ref{relationship between any tildeA and tildeh}) and (\ref{decomposition formula for tildeh in a general basis}) imply
\begin{align}\label{decomposition formula for any shear operator in a general basis}
\tildeA_\eta = \sum_{i=1}^k \g(\xi_i,\eta) \tildeA_i.
\end{align}

In the definitions that follow we introduce two useful concepts.
\begin{definition}\label{definition: shear space of S at p}
At any point $p\in \S$, the set
\begin{align*}
\Im\,\tildeh_p := \text{span} \left\{ \tildeh(v,w) : v,w\in T_p\S \right\} \subseteq T_p\S^\perp
\end{align*}
is called the \emph{shear space} of $\S$ at $p$.
\end{definition}

If $\mathscr{N}_p^1 = \text{span} \left\{ h(v,w) : v,w\in T_p\S \right\} \subseteq T_p\S^\perp$ denotes the first normal space of $\S$ at the point $p\in\S$, then for every $p$ in $\S$ we have $\Im\,\tildeh_p\subseteq\mathscr{N}_p^1$, hence $\dim\,\Im\,\tildeh_p \leq \dim\mathscr{N}_p^1\leq k$. Furthermore,  given any orthonormal basis $\{e_1,\ldots,e_n\}$ in $T_p\S$, the image of $\tildeh_p$ is spanned by the $n(n+1)/2$ vectors $\tildeh(e_i,e_j)$, for $i\leq j$. Because $\sum_{i=1}^n\tildeh(e_i,e_i)=0$, these vectors are not linearly independent. In particular, the dimension of $\Im\,\tildeh_p$ can be at most $n(n+1)/2-1$. Therefore
\begin{align}\label{dim(Im tildeh) < min(k,n(n+1)/2)}
\dim\,\Im\,\tildeh_p \leq \min \left\{ k , \frac{n(n+1)}{2} - 1 \right\}.
\end{align}

Formula (\ref{decomposition formula for any shear operator in a general basis}) for the decomposition of any shear operator implies that if $\dim\,\Im\,\tildeh_p=d$ then any $d+1$ shear operators must be linearly dependent in $p$. The converse of this is also true, so that
\begin{align}\label{dim(Im tildeh) = number of linearly independent shear operators}
\dim\,\Im\,\tildeh_p
= \max \left\{ d \, | \, \exists \, \eta_1,\ldots,\eta_d\in\normsp : \tildeA_{\eta_1}, \ldots, \tildeA_{\eta_d} \text{ are linearly independent at } p \right\}
\end{align}

Notice that similar formulas to (\ref{dim(Im tildeh) < min(k,n(n+1)/2)}) and (\ref{dim(Im tildeh) = number of linearly independent shear operators}) also hold for the dimension of the first normal space. Indeed, for the dimension of $\mathscr{N}_p^1$ we will have $\dim\,\mathscr{N}_p^1\leq\min\left\{k,n(n+1)/2\right\}$; as for (\ref{dim(Im tildeh) = number of linearly independent shear operators}), the Weingarten operators will just take the place of the shear operators. These two formulas for $\mathscr{N}_p^1$ imply that if $k-n(n+1)/2$ is positive, then there exist $k-n(n+1)/2$ linearly independent Weingarten operators that vanish at $p$. 

\begin{definition}
Assume that the dimension of the shear spaces $\Im\,\tildeh_p$ is constant on $\S$, i.e.\! there exists $d\in\mathbb{N}$ with $0\leq d\leq k$ such that $\dim\,\Im\,\tildeh_p=d$ for all $p\in\S$. The set
\begin{align*}
\Im\,\tildeh = \bigcup_{p\in\S} \Im\,\tildeh_p \subseteq \normsp
\end{align*}
is called the \emph{shear space} of $\S$. Then, the shear space is a module over the ring of functions defined on $\S$ with dimension $d$. 
\end{definition}

The properties already presented relating $\dim\,\Im\,\tildeh_p$ to the shear operators can be extended to $\dim\,\Im\,\tildeh$ accordingly.

\subsection{Umbilical points and umbilical submanifolds}

For works concerning umbilical submanifolds and some previous results in both Riemannian and semi-Riemannian settings the reader can consult, e.g.,  \cite{CSV2016} and references therein.

For hypersurfaces (co-dimension 1), a point can only be umbilical along the unique normal direction. This situation changes completely for higher co-dimensions, in which case there are multiple directions along which a point can be umbilical.

\begin{definition}
Using the notations and conventions introduced above for the immersion $\Phi: (\S,g) \to (\M,\g)$, a point $p\in\S$ is said to be
\begin{itemize}
\item \emph{umbilical with respect to} $\xi_p\in T_p\S^\perp$ if $A_{\xi_p}$ is proportional to the identity;
\item \emph{totally umbilical} if it is umbilical with respect to all $\xi_p\in T_p\S^\perp$.
\end{itemize}
\end{definition}

A point $p\in\S$ is umbilical with respect to $\xi_p\in T_p\S^\perp$ if and only if $A_{\xi_p} = (\tr\tildeA_{\xi_p}/n) \id$ or, equivalently, $\tildeA_{a \xi_p}=0$ for any $a\in \mathbb{R}\setminus \{0\}$. $\xi_p$-umbilicity is thus a property that gives information about span$\{\xi_p\}$ regardless of the length and the orientation of $\xi_p$.
Hence, we will usually state that $p$ is umbilical with respect to the normal \emph{direction} spanned by $\xi_p$. On the other hand, $p$ is totally umbilical if and only if $h(v,w)=g(v,w)H_p$ for all $v,w \in T_p\S$ or, equivalently, if and only if $\tildeh=0$ at $p$. This fact was already known for hypersurfaces in Riemannian settings and can be found, e.g., in \cite{Fialkow 1944}. However,  the relationship between the total shear tensor and the umbilical properties of submanifolds, treated in \cite{CSV2016} and in the present article, is substantially new.

\begin{definition}
Given any point $p\in \S$, the set
\begin{align*}
\U_p = \left\{ \xi_p \in T_p\S^\perp : p \text{ is umbilical with respect to } \xi_p \right\} \subseteq T_p\S^\perp
\end{align*}
is called the \emph{umbilical space} of $\S$ at $p$.
\end{definition}

\begin{lemma}\label{lemma: U(p) is a vector space}
The umbilical space $\U_p$ is a vector space for every $p\in\S$.
\end{lemma}

\begin{proof}
Let $\xi_p,\eta_p\in \U_p$, so that by definition $\tildeA_{\xi_p}=\tildeA_{\eta_p}=0$. Let $a,b\in \mathbb{R}$ and consider the normal vector $a\xi_p+b\eta_p$. By linearity we have $\tildeA_{a\xi_p+b\eta_p}=a\tildeA_{\xi_p}+b\tildeA_{\eta_p}=0$. It follows that $p$ is umbilical with respect to $a\xi_p+b\eta_p$, hence $a\xi_p+b\eta_p$ belongs to $\U_p$.
\end{proof}

It follows from Lemma \ref{lemma: U(p) is a vector space} that $\dim\,\U_p$ is well defined. Notice that $\dim\,\U_p=m$ if and only if $p$ is umbilical with respect to \emph{exactly} $m$ linearly independent normal directions. Moreover, by formulas (\ref{dim(Im tildeh) < min(k,n(n+1)/2)}) and (\ref{dim(Im tildeh) = number of linearly independent shear operators}) it follows that if $k-n(n+1)/2+1$ is positive, then $\dim\,\U_p\geq k-n(n+1)/2+1$.

\begin{definition}
Using the notations and conventions introduced above for the immersion $\Phi: (\S,g) \to (\M,\g)$, the submanifold $(\S,g)$ is said to be
\begin{itemize}
\item \emph{umbilical with respect to} $\xi\in\normsp$ if $A_{\xi}$ is proportional to the identity;
\item \emph{totally umbilical} if it is umbilical with respect to all $\xi\in\normsp$.
\end{itemize}
\end{definition}

The properties presented above for umbilical points can be extended to umbilical submanifolds accordingly. $\S$ is umbilical with respect to $\xi\in\normsp$ if and only if $\xi_p\in\U_p$ for all $p\in\S$. More in general, $\S$ is umbilical with respect to \emph{exactly} $m$ linearly independent non-zero normal vector fields $\xi_1,\ldots,\xi_m\in\normsp$ if and only if $(\xi_1)_p,\ldots,(\xi_m)_p\in\U_p$ for all $p\in\S$. Equivalently, if and only if $\dim\,\U_p=m$ for all $p\in\S$. This leads to the following definition.

\begin{definition}
Assume that the dimension of the umbilical spaces $\U_p$ is constant on $\S$, i.e.\! there exists $m\in\mathbb{N}$ with $0\leq m\leq k$ such that $\dim\,\U_p=m$ for all $p\in\S$. Then the set
\begin{align*}
\U = \left\{ \xi \in\normsp : \S \text{ is umbilical with respect to } \xi \right\} \subseteq \normsp
\end{align*}
is called the \emph{umbilical space} of $\S$.
\end{definition}

The umbilical space  of $\S$ is such that $\U=\cup_{p\in\S} \U_p$. In Lemma \ref{lemma: U(p) is a vector space} we proved that $\U_p$ is a vector space for every $p$. Similarly we can prove that $\U$ is a finitely generated module over the ring of functions defined on $\S$ with $\dim\,  \U =m$.

\section{Results}

\subsection{The relationship between $\U$ and $\Im\,\tildeh$}

\begin{proposition}\label{proposition: k-dimU=dim(im tildeh) at a point}
Let $\Phi: (\S,g) \to (\M,\g)$ be an isometric immersion of an $n$-dimensional Riemannian manifold into a semi-Riemannian manifold with co-dimension $k$. Let $\U_p$ and $\Im\,\tildeh_p$ be the umbilical space and the shear space, respectively, of $\S$ at any point $p\in\S$. Then
\begin{align*}
\U_p = (\Im\,\tildeh_p)^\perp.
\end{align*}
Moreover,
\begin{align*}
k - \dim\,\U_p = \dim\, \Im \, \tildeh_p.
\end{align*}
\end{proposition}

Here $(\Im\,\tildeh_p)^\perp$ is defined as the subspace of $T_p\S^\perp$ orthogonal to $\Im\,\tildeh_p$, namely
\begin{align*}
(\Im\,\tildeh_p)^\perp = \left\{ \eta_p\in T_p\S^\perp \, | \, \forall \xi_p\in\Im\,\tildeh_p : \g(\eta_p,\xi_p)=0 \right\}.
\end{align*}

\begin{proof}
By definition, a normal vector $\xi_p$ belongs to $\U_p$ if $\tildeA_{\xi_p}=0$. Equivalently, if $g(\tildeA_{\xi_p}(v),w)=0$ in $p$, for all $v,w\in T_p\S$. By formula (\ref{relationship between any tildeA and tildeh}), this holds if and only if $\xi_p\in(\Im\,\tildeh_p)^\perp$. Hence $\U_p=(\Im\,\tildeh_p)^\perp$.\\
Suppose $\Im\,\tildeh_p=\{0\}$, then it is clear that $\U_p=(\Im\,\tildeh_p)^\perp=T_p\S^\perp$ and the relation between the dimensions holds. Now assume $\Im\,\tildeh_p\neq\{0\}$. We can choose a basis $\{(\xi_1)_p,\ldots,(\xi_k)_p\}$ of $T_p\S^\perp$ such that $\{(\xi_1)_p,\ldots,(\xi_d)_p\}$ is a basis of $\Im\,\tildeh_p$. A normal vector $\eta_p$ belongs to $\U_p=(\Im\,\tildeh_p)^\perp$ if and only if $\g(\eta_p,(\xi_j)_p)=0$ for $j=1,\ldots,d$. Since $(\xi_1)_p,\ldots,(\xi_d)_p$ are linearly independent and $\g$ is non-degenerate, these are $d$ linearly independent conditions on the components of $\eta_p$ and hence $\dim\U_p=k-d$.
\end{proof}

By Proposition \ref{proposition: k-dimU=dim(im tildeh) at a point} we have that the umbilical space $\U_p$ and the shear space $\Im\,\tildeh_p$ are such that
\begin{align*}
\U_p=(\Im\,\tildeh_p)^\perp \qquad \text{and} \qquad \dim\,\U_p + \dim\,\Im\,\tildeh_p = k.
\end{align*}
However, the intersection $\U_p\cap\Im\,\tildeh_p$ might be non-empty, and consequently the direct sum of the two spaces does not generate, in general, the whole normal space. For example, if $p$ is umbilical with respect to some vector $\xi_p$ that is null, i.e. $\g(\xi_p,\xi_p)=0$, then $\xi_p$ might belong to $\Im\,\tildeh_p$. Actually, in case $\M$ is a \emph{Riemannian} manifold, one easily checks that $\Im\,\tildeh_p\cap\U_p=\emptyset$ and one has 
\begin{align*}
T_p\S^\perp = \Im\,\tildeh_p \oplus \U_p \hspace{2cm} (\M \, \, \mbox{Riemannian}).
\end{align*}

\begin{remark}\label{remark: constant dimension of the shear and umbilical spaces}
Proposition \ref{proposition: k-dimU=dim(im tildeh) at a point} implicitly shows that if the dimension of the shear spaces $\Im\,\tildeh_p$ is constant on $\S$ then the dimension of the umbilical spaces $\U_p$ also is, and vice versa.
\end{remark}

By Proposition \ref{proposition: k-dimU=dim(im tildeh) at a point} and Remark \ref{remark: constant dimension of the shear and umbilical spaces} follows the next corollary.

\begin{corollary}\label{corollary: k-dimU=dim(im tildeh)}
Assume that the dimension of the shear spaces $\Im\,\tildeh_p$ is constant on $\S$ (equivalently, that the dimension of the umbilical spaces $\U_p$ is constant on $\S$). Then $\Im\,\tildeh$ and $\U$ are well defined and we have
\begin{align*}
\U = (\Im\,\tildeh)^\perp.
\end{align*}
Moreover,
\begin{align*}
k - \dim\,\U = \dim\, \Im \, \tildeh.
\end{align*}
\end{corollary}

From now on, and for the sake of conciseness, we will assume that the dimension of the shear spaces $\Im\,\tildeh_p$ is constant on $\S$. We will consider that $\S$ is umbilical with respect to \emph{exactly} $m$ linearly independent umbilical directions, that is to say, $\dim\,\U=m$. Notice, however, that all results make sense also if stated pointwise.

\subsection{Characterization of umbilical spacelike submanifolds}

We will denote by $\wedge$ the wedge product of one-forms and by $\flat$ the musical isomorphism: if $V$ is a vector field on $(\M,\g)$, then its associated one-form $V^\flat$ is given by $V^\flat(Z)=\g(V,Z)$ for every vector field $Z$ on $\M$.

\begin{proposition}\label{proposition: dimU=m iff the total shear tensor wedge product is zero}
Let $\Phi: (\S,g) \to (\M,\g)$ be an isometric immersion of an $n$-dimensional Riemannian manifold into a semi-Riemannian manifold with co-dimension $k$. Let $\U$ be the umbilical space of $\S$, then $\dim\,\U=m$ if and only if the total shear tensor satisfies
\begin{align*}
{\bigwedge}^{k-m+1} \,\, \tildeh^\flat = 0
\end{align*}
with
\begin{align*}
{\bigwedge}^{k-m} \,\, \tildeh^\flat \neq 0.
\end{align*}
\end{proposition}

Here by ${\bigwedge}^q\omega$ we mean $q$ times the wedge product $\omega\wedge\cdots\wedge\omega$ of any one form $\omega$.
Notice that when we write, for instance, $\tildeh^\flat\wedge\tildeh^\flat$, we mean $\tildeh^\flat(X_1,Y_1)\wedge\tildeh^\flat(X_2,Y_2)$ for all $X_1,Y_1,X_2,Y_2\in\tangsp$.

\begin{proof}
Suppose that the dimension of the umbilical space is $m$. By Corollary \ref{corollary: k-dimU=dim(im tildeh)} it follows that $\dim\,\Im\,\tildeh=k-m$ and we can then decompose the total shear tensor by means of exactly $k-m$ normal vector fields. Explicitly, there exist $k-m$ shear operators $\{\tildeA_i\}_{i=1}^{k-m}$ and $k-m$ vector fields $\zeta_1,\ldots,\zeta_{k-m}\in\normsp$ such that
\begin{align*}
\tildeh(X,Y)
= \sum_{i=1}^{k-m} g(\tildeA_i X,Y) \zeta_i.
\end{align*}
Because these vector fields are linearly independent, their corresponding one-forms $\zeta_r^\flat$ also are. From this fact it easily follows that, on one hand, the wedge product $k-m$ times of $\tildeh^\flat$ is different from zero and, on the other hand, that the wedge product $k-m+1$ times of $\tildeh^\flat$ must be zero.

Now suppose that the total shear tensor satisfies the conditions in the statement. By algebra's basic results, we know that $l$ one-forms $\omega_1,\ldots,\omega_l$ are linearly independent if and only if their wedge product $\omega_1\wedge\cdots\wedge\omega_l$ is not zero. Equivalently, they are linearly dependent if and only if their wedge product is zero. Hence, by hypothesis, among the sets of $k$ one-forms $\{\tildeh(X_1,Y_1)^\flat,\ldots,\tildeh(X_k,Y_k)^\flat\}$ constructed for arbitrary $X_1,Y_1,\ldots,X_k,Y_k\in\tangsp$, there exist exactly $k-m$ which are linearly independent. The same holds for the corresponding normal vector fields $\{\tildeh(X_1,Y_1),\ldots,\tildeh(X_k,Y_k)\}$. By Definition \ref{definition: shear space of S at p} of shear space, this implies that the dimension of $\Im\,\tildeh$ is $k-m$. Finally by Corollary \ref{corollary: k-dimU=dim(im tildeh)}, we obtain $\dim\,\U=m$.
\end{proof}

The following theorem summarizes the results presented in Corollary \ref{corollary: k-dimU=dim(im tildeh)} and Proposition \ref{proposition: dimU=m iff the total shear tensor wedge product is zero}.

\begin{theorem}\label{main theorem}
Let $\Phi: (\S,g) \to (\M,\g)$ be an isometric immersion of an $n$-dimensional Riemannian manifold into a semi-Riemannian manifold with co-dimension $k$. Let $\U$ and $\Im\,\tildeh$ be the umbilical space and the shear space, respectively, of $\S$. Then the following conditions are all equivalent:
\begin{itemize}
\item[(i)]
the umbilical space $\U$ has dimension $m$;
\item[(ii)]
the shear space $\Im\,\tildeh$ has dimension $k-m$;
\item[(iii)]
the total shear tensor satisfies
\begin{align*}
{\bigwedge}^{k-m+1} \,\, \tildeh^\flat = 0
\end{align*}
with
\begin{align*}
{\bigwedge}^{k-m} \,\, \tildeh^\flat \neq 0;
\end{align*}
\item[(iv)]
any $k-m+1$ shear operators $\tildeA_{\xi_1},\ldots,\tildeA_{\xi_{k-m+1}}$ are linearly dependent (and there exist precisely $k-m$ shear operators that are linearly independent).
\end{itemize}
\end{theorem}

\subsection{Special cases}

Using Theorem \ref{main theorem} some particular situations are worth mentioning, for example:
\begin{itemize}
\item
If $\dim\,\U=k$ we have $\tildeh(X,Y)^\flat=0$ for every $X,Y\in\tangsp$, equivalently $\tildeh=0$, and the submanifold $\S$ is totally umbilical. In particular, $\Im\,\tildeh=\emptyset$ and $\normsp=\U$.
\item
If $\dim\,\U=k-1$ then $\dim\,\Im\,\tildeh=1$ and we have $\tildeh(X_1,Y_1)^\flat \wedge \tildeh(X_2,Y_2)^\flat=0$ for every $X_1,Y_1,X_2,Y_2\in\tangsp$. It follows that there exist a normal vector field $G\in\normsp$ and a properly normalized self-adjoint operator $\tildeA$ such that $\tildeh(X,Y)=g(\tildeA X,Y)G$ for every $X,Y\in\tangsp$. This was the case studied in \cite{CSV2016} for $k=2$.
\item
If $\dim\,\U=0$ then there are no umbilical directions and $\Im\,\tildeh=\normsp$.
\end{itemize}
Our results can be applied to all other cases too and open the door for a novel analysis of the structure of the umbilical space.

\bigskip\bigskip

\noindent
\small{\textbf{Acknowledgement:} We would like to thank the referee for comments and improvements.}


\bigskip

\begin{thebibliography}{9999}

\bibitem{CSV2016}
N.\! Cipriani, J.\! M.\! M.\! Senovilla, J.\! Van der Veken,
\emph{Umbilical properties of spacelike co-dimension two submanifolds},
Results Math. (in press): online first: doi:10.1007/s00025-016-0640-x. (arXiv:1604.06375)

\bibitem{Fialkow 1944}
A.\! Fialkow,
\emph{Conformal differential geometry of a subspace},
Trans.\! Amer.\! Math.\! Soc.\! \textbf{56} no.\! 2
(1944) 309-433

\bibitem{KobayashiNomizu}
S.\! Kobayashi, K.\! Nomizu,
\emph{Foundations of differential geometry, Volume II},
Interscience Publishers
(1969)

\bibitem{Kriele}
M.\! Kriele,
\emph{Spacetime, foundations of general relativity and differential geometry},
Springer,
Berlin
(1999)

\bibitem{ONeill}
B.\! O'Neill,
\emph{Semi-Riemannian geometry with applications to relativity},
Academic Press
(1983)

\bibitem{Senovilla 2012}
J.\! M.\! M.\! Senovilla,
\emph{Umbilical-type surfaces in spacetime},
Recent Trends in Lorentzian Geometry,
Springer Proceedings in Mathematics and Statistics
(2013), 87-109

\bibitem{Willmore 1988}
F.\! J.\! Pedit, T.\! J.\! Willmore,
\emph{Conformal geometry},
Atti Sem.\! Mat.\! Fis.\! Univ.\! Modena \textbf{36} no.\! 2
(1988) 237-245



\end{thebibliography}
\end{document}